\documentclass[11pt, reqno]{amsart}
\allowdisplaybreaks[3]

\usepackage{amsthm}
\usepackage{amscd}
\usepackage{amsmath}
\usepackage{amsfonts}
\usepackage{amssymb}
\usepackage{graphicx}
\usepackage{amsopn}
\usepackage[usenames]{color}
\usepackage{tikz}
\usepackage{textcomp}
\usetikzlibrary{arrows,shapes,snakes,automata,backgrounds,petri}

\theoremstyle{theorem}
\newtheorem{theorem}{Theorem}

\theoremstyle{definition}
\newtheorem{definition}[theorem]{Definition}

\newtheorem{lemma}[theorem]{Lemma}

\newtheorem{remark}[theorem]{Remark}

\def\F{{\mathbb{F}}}

\def\char{{\rm char \,}}
\def\per{{\rm per}\,}
\def\prk{{\rm prk}\,}

\def\Mat{{\rm Mat}}

\def\dim{{\rm dim\, }}
\def\diag{{\rm diag}}

\def\rk{{\rm rk}\,}

\numberwithin{theorem}{section}

\begin{document}

	\title{On linear preservers of permanental rank}
	
	\thanks{The second author is supported by ISF Moked grant 2919/19.}
	
	\author{A. E. Guterman}
	\address{A. E. Guterman: \newline Department of Mathematics, Bar-Ilan University, Ramat-Gan 5290002, Israel; \newline 
		Moscow Center for Fundamental and Applied Mathematics, Moscow 119991, Russia}
	\email{alexander.guterman@biu.ac.il}

	\author{I. A. Spiridonov}
	\address{I. A. Spiridonov: \newline Weizmann Institute of Science, Rehovot 7610001, Israel; \newline Moscow Center for Fundamental and Applied Mathematics, Moscow 119991, Russia}
	\email{spiridonovia@ya.ru}

	\keywords{Permanent, Rank, Linear map, Preservers.}
	
	\subjclass{15A86, 05C50, 15A15, 47L05.}
	
	\maketitle
	\thispagestyle{empty}
	
	\vspace{0.5cm}

	\begin{abstract}
		Let $\Mat_{n}(\F)$ denote the set of square $n\times n$ matrices over a field $\F$ of characteristic different from two. The permanental rank $\prk(A)$ of a matrix $A \in \Mat_{n}(\F)$ is the size of the maximal square submatrix in $A$ with nonzero permanent. By $\Lambda^{k}$ and $\Lambda^{\leq k}$ we denote the subsets of matrices $A \in  \Mat_{n}(\F)$ with $\prk(A) = k$ and $\prk(A) \leq k$, respectively. In this paper for each $1 \leq k \leq n-1$ we obtain a complete characterization of  linear maps $T: \Mat_{n}(\F) \to \Mat_{n}(\F)$ satisfying $T(\Lambda^{\leq k}) = \Lambda^{\leq k}$ or bijective linear maps satisfying $T(\Lambda^{\leq k}) \subseteq \Lambda^{\leq k}$. Moreover, we show that if  $\F$ is an infinite field, then $\Lambda^{k}$ is Zariski dense in $\Lambda^{\leq k}$ and apply this to describe such bijective linear maps   satisfying $T(\Lambda^{k}) \subseteq \Lambda^{k}$.
	\end{abstract}
	
	\section{Introduction}
	
	Let $\Mat_{n}(\F)$ denote the set of $n \times n$ matrices over a field $\F$ with $\char (\F) \neq 2$; throughout the paper we assume that $n \geq 3$. The theory of linear operators which map certain special sets or functions on matrices
	into themself, so-called linear preservers, has been intensively investigated for more than 100 years. This theory was started in 1897 with the work of Frobenius \cite{Frobenius}, where the complete description of determinant-preserving linear maps $T: \Mat_{n}(\F) \to \Mat_{n}(\F)$ was obtained.
	In 1949, Dieudonne \cite{Dieudonne} received a complete characterization of bijective linear maps $T: \Mat_{n}(\F) \to \Mat_{n}(\F)$ preserving the set of singular matrices. In 1959, Marcus and Moyls \cite{MarcusMoyls} classified linear maps of $\Mat_{n}(\F)$ preserving the rank function for algebraically closed fields of characteristic zero and described  rank-$1$ preserving maps, see \cite{MarcusMoyls2}. In the same year, Marcus and Purves \cite{MarcusPurves} described linear preservers of symmetric functions on matrix eigenvalues, in particular, they classified linear maps preserving the set of invertible matrices. Then Westwick \cite{Westwick} generalized results on linear rank-preserving maps for arbitrary algebraically closed fields.
	
	Later, the classification problem of rank-$k$ preserving linear maps was studied for fixed $k$, where $1 \leq k \leq n$.
	Beasley in a series of papers  \cite{Beasley70, Beasley81, Beasley83, Beasley} studied nonsingular linear maps $T: \Mat_{n}(\F) \to \Mat_{n}(\F)$ preserving the set of matrices of rank $k$, under some conditions on $k$ and $\F$.   Beasley and Laffey obtained a complete description of nonsingular linear  rank-$k$ preservers provided that $\F$ has at least four elements.
	Chan and Lim \cite{ChanLim} obtained a description of linear maps preserving the set of nonzero matrices of rank less or equal to $k$ using the results of Flanders \cite{Flanders} on maximal subspaces of rank-bounded matrices.
	Loewy \cite{Loewy} obtained a classification of rank-$k$ nonincreasing linear maps $T: \Mat_{n}(\F) \to \Mat_{n}(\F)$ for algebraically closed fields $\F$.
	Linear preservers of some more general functions on different matrix subspaces were also investigated.  For the survey of results in this area see ~\cite{Li92, Molnar, Pierce, Semrl06} and references therein.

	In parallel to determinant-preserving and rank-preserving problems, there are similar problems related to the matrix permanent function instead of the determinant.
	Marcus and May \cite{Marcus} obtained the complete characterization of the linear maps $T: \Mat_{n}(\F) \to \Mat_{n}(\F)$, satisfying $\per(T(A)) = \per(A)$. This result was extended and generalized in many further works, see for example, \cite{coelho2006} and its references. In \cite{GutSp} the authors proved that any additive surjective map, satisfying $\per(T(A)) = \per(A)$, is necessary linear. A natural problem is to study maps preserving the analog of rank based on the permanent function, so called \textit{permanental rank}.
	
	\begin{definition}
		Let $A \in \Mat_{n}(\F)$.  Permanental rank  $\prk(A)$ of the matrix  $A$ is the size of the maximal square submatrix in  $A$ 
		with nonzero permanent. 
	\end{definition}

	The permanental rank of a matrix is a classical invariant firstly introduced in \cite{Yu}. It has a lot of applications in combinatorics and linear algebra. A famous example is the Alon-Jaeger-Tarsi conjecture on nowhere-zero points of linear maps \cite{Alon}, which can be reformulated via the notion of permanental rank, and is recently partially proved in \cite{Nagy}. Also, permanental rank is useful in the study of other topics related to the permanent function, see, for example,~\cite{GutSp}.

	In this paper we study linear maps of $\Mat_n(\F)$ preserving the sets of matrices with fixed permanental rank and with permanental rank bounded from above. Let us denote these sets as follows
	$$\Lambda^{k} =\{A\in \Mat_{n}(\F) \; \vert \; \prk(A) = k\},$$
	$$\Lambda^{\leq k} =\{A\in \Mat_{n}(\F) \; \vert \; \prk(A) \leq k\}.$$
	Note that, unlike $\Lambda^k$, the set $\Lambda^{\leq k}$ is Zariski closed in $\Mat_n(\F)$.
	For each $1 \leq k \leq n-1$ we obtain a complete characterization of linear bijective maps $T: \Mat_{n}(\F) \to \Mat_{n}(\F)$ satisfying $T(\Lambda^{\leq k}) \subseteq \Lambda^{\leq k}$. Namely, we show that in this case $T$ is a composition of row (column) permutations, row (column) rescalings, and 
	transposition, see Theorem \ref{mainth1}. In the case where $\F$ is an infinite field, we obtain the similar result for linear bijective maps $T$ satisfying $T(\Lambda^{k}) \subseteq \Lambda^{k}$, see Theorem~\ref{mainth12}. Informally, this means that, unlike determinantal varieties, the varieties $\Lambda^{k}$ and $\Lambda^{\leq k}$ have very few symmetries.
	
	Our results are parallel to results of Beasley and Laffey \cite{BeasleyLaffey90} and Loewy \cite{Loewy}, where the similar questions were studied for the usual rank function. Our approach is to use the action of $T$ on the set of maximal subspaces in $\Lambda^{\leq k}$; the classification of such subspaces was recently obtained in \cite{GutMeshSp}. This approach dates back to the one used for usual rank in the classical results mentioned above. However, to solve the corresponding problem for permanental rank we transfer it to the graph-theoretical framework as follows. We consider maximal subspaces in $\Lambda^{\leq k}$ as the vertices of a complete weighted graph $\Theta_{n, k}$.  The weights are equal to the dimensions of the intersections of the subspaces in the vertices. Then we study the action of $T$ on this graph by automorphisms and finally show that either $T$ or its composition with the transposition maps rows to rows and columns to columns. As a consequence we characterize $T$.  We use  algebraic geometry approach  to extend our characterization to $\Lambda^{k}$-preservers.

	For a permutation $ \sigma \in S_n $, we denote by $P(\sigma)=(p_{ij})$ the permutation matrix corresponding to $\sigma$, i.e., $$p_{ij}=\left\{ \begin{array}{cl} 1, & \mbox{ if } j=\sigma^{-1}(i) \\ 0, & \mbox{ otherwise } \end{array} \right. $$
	The main result of the present paper is as follows.

	\begin{theorem}\label{mainth1} 
		Let $\F$ be an arbitrary field with $\char (\F) \neq 2$. Let
		$1 \leq k \leq n-1$, and $T: \Mat_{n}(\F) \to \Mat_{n}(\F)$ be a linear bijective map. Then $T$ satisfies $T(\Lambda^{\leq k}) \subseteq \Lambda^{\leq k}$ if and only if there exist permutations $\sigma_1, \sigma_2 \in S_n$ and diagonal matrices $D_1, D_2 \in \Mat_n(\F)$ with nonzero diagonal elements such that either
		\begin{equation}\label{eq1}
			T(A) = D_1 P(\sigma_1) A P(\sigma_2) D_2 \;\;\; \mbox{for all} \;\;\; A \in \Mat_n(\F)
		\end{equation}
		or
		\begin{equation}\label{eq2}
			T(A) = D_1 P(\sigma_1) A^T P(\sigma_2) D_2 \;\;\; \mbox{for all} \;\;\; A \in \Mat_n(\F).
		\end{equation}
	\end{theorem}
	
	Let us note that the bijectivity condition in Theorem \ref{mainth1} is necessary, since otherwise the class of such maps would be significantly larger. For example, any linear map with the image lying inside $\Lambda^{\leq k}$ (we will see in Theorem \ref{th2} that there is a large family of vector subspaces inside $\Lambda^{\leq k}$) satisfies the condition $T(\Lambda^{\leq k}) \subseteq \Lambda^{\leq k}$, but is not of the form (\ref{eq1}) or (\ref{eq2}).
	
	However, it turns out that if we replace $T(\Lambda^{\leq k}) \subseteq \Lambda^{\leq k}$ by a stronger condition $T(\Lambda^{\leq k}) = \Lambda^{\leq k}$, then the bijectivity of $T$ is not necessary. The second form of the main result is as follows.
	
	\begin{theorem}\label{mainth} Let $\F$ be an arbitrary field  with $\char (\F) \neq 2$. Let $1 \leq k \leq n-1$, and $T: \Mat_{n}(\F) \to \Mat_{n}(\F)$ be a linear map. Then $T$ satisfies $T(\Lambda^{\leq k}) = \Lambda^{\leq k}$ if and only if there exist permutations $\sigma_1, \sigma_2 \in S_n$ and diagonal matrices $D_1, D_2 \in \Mat_n(\F)$ with nonzero diagonal elements such that either (\ref{eq1}) or (\ref{eq2}) holds.
	\end{theorem}
	
	For infinite $\F$, we obtain the similar classification for $\Lambda^k$-preservers.
	
	\begin{theorem}\label{mainth12}
		Let $\F$ be an infinite field with $\char (\F) \neq 2$.
		Let $1 \leq k \leq n-1$ and let $T: \Mat_{n}(\F) \to \Mat_{n}(\F)$ be a linear bijective map. Then $T$ satisfies $T(\Lambda^{k}) \subseteq \Lambda^{k}$ if and only if there exist permutations $\sigma_1, \sigma_2 \in S_n$ and diagonal matrices $D_1, D_2 \in \Mat_n(\F)$ with nonzero diagonal elements such that either (\ref{eq1}) or (\ref{eq2}) holds.
	\end{theorem}
	
	The key argument in the proof of Theorem \ref{mainth12} is Lemma~\ref{density}, reducing the problem to $\Lambda^{\leq k}$-preservers via the proof that $\Lambda^k$ is Zariski dense in $\Lambda^{\leq k}$.

	This paper is organized as follows. In Section \ref{sec1.5} we 
	deduce Theorems \ref{mainth} and \ref{mainth12} from Theorem \ref{mainth1}. The remainder of this paper is devoted to the proof of Theorem \ref{mainth1}. Section \ref{sec2} contains some auxiliary results from \cite{GutMeshSp} about the structure of maximal subspaces in $\Lambda^{\leq k}$. Then in Section \ref{sec3} we construct a complete weighted graph $\Theta_{n, k}$, described in Definition \ref{def:Tnk} 
	and study its automorphisms.
	Finally, in Section \ref{sec4} we provide a proof of Theorem \ref{mainth1}.
	
	This work was carried out during the stay of the second author as a visiting student at the Weizmann Institute of Science. He is grateful to U.~Bader for his hospitality during the visit and to A.~Zaitsev for useful comments on algebraic geometry.
	
	\section{Auxiliary results} \label{sec1.5}
	
	First we provide the following result which proof is straightforward.
	\begin{lemma}\label{preser}
		The permanental rank function on $\Mat_n(\F)$ is invariant w.r.t. the following matrix operations.
		
		(1) Transposition, i.e. $A \mapsto A^T$.
		
		(2) Row permutations, i.e. $A \mapsto P(\sigma) A$, where $\sigma \in S_n$.
		
		(3) Column permutations, i.e. $A \mapsto A P(\sigma)$, where $\sigma \in S_n$.
		
		(4) Row nonzero rescaling, i.e. $A \mapsto D A$, where $D = \diag(d_1, \dots, d_n) \in \Mat_n(\F)$ is a diagonal matrix with nonzero diagonal entries $d_i \neq 0$.
		
		(5) Column nonzero rescaling, i.e. $A \mapsto A D$, where $D = \diag(d_1, \dots, d_n) \in \Mat_n(\F)$ is a diagonal matrix with nonzero diagonal entries $d_i \neq 0$.
	\end{lemma}
	
	Now we can prove the ``if'' directions of Theorems \ref{mainth1}, \ref{mainth} and \ref{mainth12}.
	
	\begin{lemma}\label{if1}
		Let $\sigma_1, \sigma_2 \in S_n$ and let $D_1, D_2 \in \Mat_n(\F)$ be diagonal matrices with nonzero diagonal elements.
		Let $T: \Mat_{n}(\F) \to \Mat_{n}(\F)$ be a map defined by either
		\begin{equation*}\label{main1}
			T(A) = D_1 P(\sigma_1) A P(\sigma_2) D_2 \;\;\; \mbox{for all} \;\;\; A \in \Mat_n(\F)
		\end{equation*}
		or
		\begin{equation*}\label{main2}
			T(A) = D_1 P(\sigma_1) A^T  P(\sigma_2) D_2 \;\;\; \mbox{for all} \;\;\; A \in \Mat_n(\F).
		\end{equation*}
		Then $T$ is bijective and $T(\Lambda^{k}) = \Lambda^{k}$ for any $0 \leq k \leq n-1$. In particular, $T(\Lambda^{\leq k}) = \Lambda^{\leq k}$.
	\end{lemma}
	\begin{proof}
		In both cases the map $T$ is given by a composition of the operations (1)-(5) from Lemma \ref{preser}, so it preserves permanental rank. In particular, $T(\Lambda^{k}) = \Lambda^{k}$. Since the matrices $D_1, D_2, P(\sigma_1), P(\sigma_2)$ are invertible, then $T$ is bijective.
	\end{proof}
	
	Let us deduce Theorem \ref{mainth} from Theorem \ref{mainth1}.
	Denote by $E_{i,j}$ the $(i,j)$th matrix unit, i.e., the matrix with 1 in $(i,j)$th position and $0$ elsewhere.
	We need the following lemma.
	
	\begin{lemma}\label{iso}
		Let $1 \leq k \leq n-1$ and let $T: \Mat_{n}(\F) \to \Mat_{n}(\F)$ be a linear map satisfying $T(\Lambda^{\leq k}) = \Lambda^{\leq k}$. Then $T$ is bijective.
	\end{lemma}
	\begin{proof}
		Since $\prk(E_{i, j}) = 1$ for any $i$ and $j$, then we have $E_{i, j} \in \Lambda^{\leq k} = T(\Lambda^{\leq k})$. Therefore the image of $T$ is a linear subspace of $\Mat_n(\F)$ containing the matrices $E_{i, j}$ for all $i, j$. Since these matrices form a basis of $\Mat_n(\F)$, then $T$ is surjective and, therefore, bijective.
	\end{proof}
	Now we can deduce Theorem \ref{mainth} from Theorem \ref{mainth1}.
	\begin{proof}[Proof of Theorem \ref{mainth}]
		By Lemma \ref{if1} it is sufficient to prove ``only if'' direction.
		Lemma \ref{iso} implies that $T$ is bijective. Therefore Theorem \ref{mainth1} is applicable, which implies the result.
	\end{proof}
	
	In order to deduce Theorem \ref{mainth12} from Theorem \ref{mainth1} we need to recall the definition of the Zariski topology on the affine space $\F^m$.

	\begin{definition}
		The \emph{Zariski topology} on $\F^m$ is defined by its basis consisting of the \emph{principal open sets}
		$$D(f) = \{x \in\F^m \; | \; f(x) \neq 0\},$$
		where $f \in \F[x_1, \dots, x_m]$ is an arbitrary polynomial.
	\end{definition}
	
	Since $D(f) \cap D(g) = D(fg)$, it follows that this system is closed under finite intersections and the topology is well defined.
	Note that any polynomial (in particular, linear) map is continuous w.r.t. the Zariski topology. Indeed, the preimage of the principal open set $D(f)$ under a polynomial map $G: \F^m \to \F^m$ is precisely the principal open set $D(f \circ G)$.
	
	Now we consider the space $\Mat_n(\F)$ as an affine space $\F^{n^2}$. Then the Zariski topology on $\Mat_n(\F)$ has a basis consisting of the principal open sets
	$$D(f) = \{A \in \Mat_n(\F) \; | \; f(A) \neq 0\},$$
	where $f$ is a polynomial on matrix entries with coefficients in $\F$.

	Let $A \in \Mat_{n}(\F)$, $J_1,J_2\subseteq [n]$. By $A[J_1\vert J_2]$ we denote the $|J_1|\times |J_2|$ submatrix of $A$ lying on the intersection of the rows with the indices from $J_1$ and the columns with the indices from $J_2$. 
	We start with a weaker statement.
	\begin{lemma}\label{density2}
		Let $\F$ be an infinite field and let $1 \leq k \leq n$. Then $\Lambda^k$ is Zariski dense in $\Lambda^{k} \cup \Lambda^{k-1}$.
	\end{lemma}
	\begin{proof}
		It suffices to check that for any $A \in \Lambda^{k-1}$ and for any its basis open neighborhood $D(f)$ there exists $X \in \Lambda^{k}$ such that $X \in D(f)$. Let us fix any $A \in \Lambda^{k-1}$ and any polynomial $f$ satisfying $f(A) \neq 0$.
		
		Since $\prk(A) = k-1$ there exist $I,J\subseteq [n]$ with $|I| = |J| = k-1$ such that $\per(A[I \vert J]) \neq 0$. Fix any $i \notin I$ and $j \notin J$. For each $\mu \in \F$ we define  $B(\mu) = A + \mu E_{i, j} \in \Mat_n(\F)$. Consider the polynomial $p$ in one variable defined by $p(\mu) = f(B(\mu))$ for $\mu \in \F$. We have
		$$p(0) = f(B(0)) = f(A) \neq 0,$$
		so $p$ is nonzero. Therefore, since $F$ is infinite, there exist infinite number of $\mu$ such that $p(\mu) \neq 0$.
		
		Now let us consider the polynomial $q$ in one variable defined by $$q(\mu) = \per(B(\mu)[I \cup \{i\} \vert J \cup \{j\}]),$$
		note that $q$ has the degree less than or equal to $1$ as the polynomial in $\mu$ since there is only one indeterminate $\mu$ in the matrix $B(\mu)$ located in the position $(i,j)$. On the other hand, the coefficient at $\mu$ in $q(\mu)$ equals to $\per(A[I \vert J]) \neq 0$. Therefore,  the degree of $q(\mu)$ is precisely $1$ and $q(\mu) = 0$ for precisely one $\mu \in \F$.
		
		Consequently, there exists $\widehat{\mu} \in \F$ such that 
		$p(\widehat{\mu}) \neq 0$ and $q(\widehat{\mu}) \neq 0$. 
		Now we are ready to construct the matrix $X \in \Lambda^{k}$ such that $X \in D(f)$.
		Define $X = (x_{i, j}) = B(\widehat{\mu})$. We have
		$$f(X) = f(D(\widehat{\mu})) = p(\widehat{\mu}) \neq 0,$$
		so $X \in D(f)$. Also, we obtain
		$$\per(X[I \cup \{i\} \vert J \cup \{j\}]) = \per(B(\widehat{\mu})[I \cup \{i\} \vert J \cup \{j\}]) = q(\widehat{\mu}) \neq 0,$$
		so $\prk(X) \geq k$.
		
		It remains to show that $\prk(X) \leq k$. Assume the converse. Then there exist $I',J'\subseteq [n]$ with $|I'| = |J'| = k+1$ such that $\per(X[I' \vert J']) \neq 0$.  Note that $A$ and $X$ differs only in the $(i, j)$-th entry. So if $i \notin I'$ or $j \notin J'$, then $X[I' \vert J'] = A[I' \vert J']$, however, $\prk(A) = k-1$, but $\per(X[I' \vert J']) \neq 0$, which is a contradiction. Therefore, we have $i \in I'$ and $j \in J'$. Consider the decomposition of $\per(X[I' \vert J'])$ by the $i$-th row.
		\begin{equation} \label{summm}
			0 \neq \per(X[I' \vert J']) = \sum_{d \in J'} x_{i, d} \per(X[I' \setminus \{i\} \vert J' \setminus \{d\}]).
		\end{equation}
		Since $A$ and $X$ differs only in the $(i, j)$-th entry, then for all $d$ we have $$X[I' \setminus \{i\} \vert J' \setminus \{d\}] = A[I' \setminus \{i\} \vert J' \setminus \{d\}].$$ Hence
		$$\per(X[I' \setminus \{i\} \vert J' \setminus \{d\}] = \per(A[I' \setminus \{i\} \vert J' \setminus \{d\}] = 0$$
		since $\prk(A) = k-1$. Therefore each term in (\ref{summm}) is zero, so we obtain a contradiction.
		
		Consequently, we have $\prk(X) = k$ and $X \in D(f)$, which implies the result.
	\end{proof}
	
	The following lemma is a direct consequence of the previous one.
	
	\begin{lemma} \label{density}
		Let $\F$ be an infinite field and let $1 \leq k \leq n$. Then $\Lambda^k$ is Zariski dense in $\Lambda^{\leq k}$.
	\end{lemma}

	\begin{proof} 
		Let us fix $1 \leq k \leq n$. By Lemma \ref{density2} we have that $\Lambda^s$ is dense in $\Lambda^{s-1} \cup \Lambda^{s}$ for each $1 \leq s \leq k$, which implies that $\Lambda^{s-1} \subseteq \overline{\Lambda^s}$, where by $\overline{\Lambda^s}$ we denote the closure of $\Lambda^s$ w.r.t. the Zariski topology. Since $\overline{\Lambda^s}$ is closed, this implies that $\overline{\Lambda^{s-1}} \subseteq \overline{\Lambda^s}$. We obtain
		$$\overline{\Lambda^0} \subseteq \overline{\Lambda^1} \subseteq \dots \subseteq \overline{\Lambda^{k-1}} \subseteq \overline{\Lambda^k}.$$
		Therefore we have $\Lambda^{\leq k}=\Lambda^0\cup \Lambda^1 \cup \ldots \cup \Lambda^{k-1}\cup \Lambda^k\subseteq \overline{\Lambda^k}$, which implies that $\Lambda^k$ is Zariski dense in~$\Lambda^{\leq k}$.
	\end{proof}

	Now we can deduce Theorem \ref{mainth12} from Theorem \ref{mainth1}.
	\begin{proof}[Proof of Theorem \ref{mainth12}]
		By Lemma \ref{if1} it is sufficient to prove ``only if'' direction. Let $T: \Mat_{n}(\F) \to \Mat_{n}(\F)$ be a linear bijective map satisfying $T(\Lambda^{k}) \subseteq \Lambda^{k}$. Since $T$ is linear, it is continuous w.r.t. Zariski topology.
		
		By Lemma \ref{density} $\Lambda^k$ is dense in $\Lambda^{\leq k}$, i.e., $\Lambda^{\leq k}\subseteq \overline{\Lambda^{k}}$. Since $\Lambda^{\leq k}$ is defined by a system of polynomial equations, it is Zariski closed in $\Mat_n(\F)$. Therefore, $\overline{\Lambda^{ k}}\subseteq \overline{\Lambda^{\leq k}}=\Lambda^{\leq k}$. Hence we have $\overline{\Lambda^{k}} = \Lambda^{\leq k}$, where by $\overline{\Lambda^{k}}$ we denote the closure of $\Lambda^{k}$ is Zariski topology. Therefore, by continuity of $T$ the condition $T(\Lambda^{k}) \subseteq \Lambda^{k}$ implies that 
		$$T(\Lambda^{\leq k}) = T(\overline{\Lambda^{k}}) \subseteq \overline{T(\Lambda^{k})} \subseteq \overline{\Lambda^{k}} = \Lambda^{\leq k}.$$
		So the result follows from Theorem~\ref{mainth1}.
	\end{proof}

	\begin{remark} \label{detvar}
		Note that the analog of Lemma \ref{density} for the usual rank is also true.
		For algebraically closed fields $\F$ this immediately follows from the fact that the determinantal variety is irreducible \cite[Corollary 4.13]{Kleiman}, so each of its nonempty Zariski open subsets is Zariski dense. However, it is an open question whether $\Lambda^{\leq k}$ is irreducible, and we suggest that the answer is negative. In this connection we refer the reader to \cite[Lemma 2.3]{GutSp} where a complete description of $\Lambda^{1}$ is obtained. It follows from that description that $\Lambda^{\leq 1}$ is not irreducible.
	\end{remark}
	
	\section{Maximal subspaces in $\Lambda^{\leq k}$} \label{sec2}
	
	The main goal of the remainder of this paper is to prove Theorem \ref{mainth1}.
	In this section we provide some results from \cite{GutMeshSp} about maximal subspaces in $\Lambda^{\leq k}$.
	
	By ${[n]}\choose{k}$ we denote the set of $k$-element subsets of $[n] = \{1,2,\ldots,n\}$; the number of such subsets equals to the binomial coefficient ${n}\choose{k}$. For each $S \in {{[n]}\choose{k}}$ let us define the subspaces $V_S^{row}, V_S^{col} \subset \Mat_n(\F)$ as follows.
	$$V_S^{row} = \{A = (a_{i, j}) \in \Mat_n(\F) \; | \; a_{i, j} = 0 \mbox{ if } i \notin S\},$$
	$$V_S^{col} = \{A = (a_{i, j}) \in \Mat_n(\F) \; | \; a_{i, j} = 0 \mbox{ if } j \notin S\}.$$
	In other words, $V_S^{row}$ ($V_S^{col}$) consists of matrices with all nonzero entries located in the rows (columns) with numbers from $S$.
	
	\begin{theorem}\cite[Corollary 1.6]{GutMeshSp}\label{th1}
		Let $V \subseteq \Lambda^{\leq k}$ be a vector subspace. Then $\dim V \leq kn$.
	\end{theorem}
	
	\begin{theorem}\cite[Theorem 1.7]{GutMeshSp}\label{th2}
		Let $V \subseteq \Lambda^{\leq k}$ be a vector subspace. Then $\dim V = kn$ if and only if $V = V_S^{row}$ or $V = V_S^{col}$ for some $S \in {{[n]}\choose{k}}$.
	\end{theorem}
	
	By Theorem \ref{th1} the dimension of subspaces in $\Lambda^{\leq k}$ is bounded from above by $nk$. So the following definition is natural.
	\begin{definition}
		A subspace $V \subseteq \Mat_n(\F)$ is called \textit{maximal} if $V \subseteq \Lambda^{\leq k}$ and $\dim V = kn$.
	\end{definition}
	
	We obtain the following result, which will be the main tool the next section.
	
	\begin{lemma}\label{cor1}
		Let $1 \leq k \leq n-1$ and let $T: \Mat_{n}(\F) \to \Mat_{n}(\F)$ be a bijective linear map satisfying $T(\Lambda^{\leq k}) \subseteq \Lambda^{\leq k}$.
		
		(1) A subspace $V$ is maximal if and only if $V = V_S^{row}$ or $V = V_S^{col}$ for some $S \in {{[n]}\choose{k}}$.
		
		(2) If $V$ is maximal then $T(V)$ is maximal.
	\end{lemma}
	\begin{proof}
		Statement (1) immediately follows from Theorem \ref{th2}. In order to prove (2), consider a maximal subspace $V \subseteq \Lambda^{\leq k}$, i.e. $\dim V = kn$.
		Since $T(\Lambda^{\leq k}) \subseteq \Lambda^{\leq k}$, then $T(V) \subseteq \Lambda^{\leq k}$. The bijectivity of $T$ implies that $\dim T(V) = \dim V$, so $\dim T(V) = \dim V = kn$ and hence $T(V)$ is maximal.
	\end{proof}

	\section{Maximal subspace graph}\label{sec3}
	
	In this section we define and study a graph $\Theta_{n, k}$ called \textit{maximal subspace graph} which plays the central role in the proof of the main result.
	
	\begin{definition} \label{def:Tnk}
		The \textit{maximal subspace graph} $\Theta_{n, k}$ is a complete weighted unoriented graph, whose vertices are maximal subspaces $V \subseteq \Mat_n(\F)$ and the weight $w(U, V) = w(V, U)$ of the edge between $U$ and $V$ equals to $\dim(U \cap V)$.
	\end{definition}
	
	We have the following result.
	\begin{lemma}\label{lem1}
		Let $1 \leq k \leq n-1$.
		
		(1) The vertices of $\Theta_{n, k}$ are the subspaces $V_S^{row}$ and $V_S^{col}$, where $S \in {{[n]}\choose{k}}$. In particular, the number of vertices of $\Theta_{n, k}$ equals $2{{n}\choose{k}}$.
		
		(2) Let $S, S' \in {{[n]}\choose{k}}$. Then
		$$w(V_{S}^{row}, V_{S'}^{col}) = k^2, \;\;\;\; w(V_{S}^{row}, V_{S'}^{row}) = w(V_{S}^{col}, V_{S'}^{col}) = n|S \cap S'|.$$
	\end{lemma}
	\begin{proof}
		Claim (1) immediately follows from Lemma \ref{cor1} (1), so let us check  claim (2). We have
		$$\dim(V_{S}^{row} \cap V_{S'}^{col}) = \dim (\langle E_{i, j} \; | \; i \in S, j \in S'\rangle_\F) = |S|\cdot |S'| = k^2,$$
		$$\dim(V_{S}^{row} \cap V_{S'}^{row}) = \dim (\langle E_{i, j} \; | \; i \in S \cap S'\rangle_\F) = n|S \cap S'|,$$
		$$\dim(V_{S}^{col} \cap V_{S'}^{col}) = \dim (\langle E_{i, j} \; | \; j \in S \cap S'\rangle_\F) = n|S \cap S'|.$$
		This concludes the proof.
	\end{proof}
	
	We say that a map $\Phi$ defined on the vertex set of $\Theta_{n, k}$ is an automorphism if $\Phi$ is  a bijection and preserves weights, i.e. $w(\Phi(u), \Phi(v)) = w(u, v)$ for any pair of vertices $u, v$ of $\Theta_{n, k}$.
	The key observation is as follows.
	\begin{lemma}\label{l2}
		Let $1 \leq k \leq n-1$ and let $T: \Mat_{n}(\F) \to \Mat_{n}(\F)$ be a linear bijective map satisfying $T(\Lambda^{\leq k}) \subseteq \Lambda^{\leq k}$.
		Then $T$ induces an automorphism $T_*$ of $\Theta_{n, k}$ given by $T_*: V \mapsto T(V)$, where $V$ is a maximal subspace representing a vertex of $\Theta_{n, k}$.
	\end{lemma}
	\begin{proof}
		Lemma \ref{cor1} implies that maximal subspaces are mapped to maximal subspaces, so $T_*$ is a well defined map on vertices of $\Theta_{n, k}$. Moreover, it is a bijection since $T$ is a bijection, and we have that $T(U \cap V) = T(U) \cap T(V)$ for any subspaces $U$ and $V$, so 
		$$w(T(U), T(V)) = \dim(T(U) \cap T(V)) = \dim(U \cap V) = w(U, V).$$
		Therefore $T_*$ preserves the weights of edges.
	\end{proof}
	
	Denote by $\Theta_{n, k}^{row}$ the subgraph of $\Theta_{n, k}$ formed by the vertices $V_S^{row}$ for all $S \in {{[n]}\choose{k}}$ and all edges between them (with the same weights). Similarly we define the subgraph $\Theta_{n, k}^{col}$. The main result of this section is as follows.
	
	\begin{theorem} \label{autom}
		Let $1 \leq k \leq n$ such that $(k, n) \neq (2, 4)$ and let $\Phi$ be an automorphism of $\Theta_{n, k}$. Then either $\Phi(\Theta_{n, k}^{row}) = \Theta_{n, k}^{row}$ and $\Phi(\Theta_{n, k}^{col}) = \Theta_{n, k}^{col}$ or $\Phi(\Theta_{n, k}^{row}) = \Theta_{n, k}^{col}$ and $\Phi(\Theta_{n, k}^{col}) = \Theta_{n, k}^{row}$.
	\end{theorem}
	\begin{proof}
		Define the subgraph $\widehat{\Theta}_{n, k} \subseteq \Theta_{n, k}$ as follows. The vertex set of $\widehat{\Theta}_{n, k}$ coincides with the vertex set of $\Theta_{n, k}$, and the edge $(u, v)$ belongs to $\widehat{\Theta}_{n, k}$ if and only if $w(u, v) = n(k-1)$. Note that $n(k-1)$ is the maximal possible weight on edge between two vertices in $\Theta_{n, k}^{row}$ or $\Theta_{n, k}^{col}$.
		We consider $\widehat{\Theta}_{n, k}$ as a graph without weights on edges.
		
		Let $\widehat{\Theta}_{n, k}^{row}$ be the subgraph of $\widehat{\Theta}_{n, k}$ formed by the vertices $V_S^{row}$ for all $S \in {{[n]}\choose{k}}$ and all edges between them. Similarly we define the subgraph $\widehat{\Theta}_{n, k}^{col}$.
		Let $\Phi$ be an automorphism of $\Theta_{n, k}$. We claim that $\Phi$ is also an automorphism of $\widehat{\Theta}_{n, k}$. This follows from the fact the $\Phi$ preserves the set of edges with weight $n(k-1)$. Therefore to prove the theorem, it is suffices to show that either $\Phi(\widehat{\Theta}_{n, k}^{row}) = \widehat{\Theta}_{n, k}^{row}$ and $\Phi(\widehat{\Theta}_{n, k}^{col}) = \widehat{\Theta}_{n, k}^{col}$ or $\Phi(\widehat{\Theta}_{n, k}^{row}) = \widehat{\Theta}_{n, k}^{col}$ and $\Phi(\widehat{\Theta}_{n, k}^{col}) = \widehat{\Theta}_{n, k}^{row}$.

		We claim that in this case $\widehat{\Theta}_{n, k}$ has precisely two connected components: $\widehat{\Theta}_{n, k}^{row}$ and $\widehat{\Theta}_{n, k}^{col}$. First let us show that there are no edges between $\widehat{\Theta}_{n, k}^{row}$ and $\widehat{\Theta}_{n, k}^{col}$. Since all these edges in $\Theta_{n, k}$ have weight $k^2$, we need to check that $k^2 \neq n(k-1)$. For $k > 2$ this is true because $k^2$ and $k-1$ are coprime (and $k-1 \neq 1$). For $k = 2$ it is equivalent to $n \neq 4$, and we have the condition $(k, n) \neq (2, 4)$. For $k = 1$ it is straightforward.
		
		Finally, let us check that $\widehat{\Theta}_{n, k}^{row}$ and $\widehat{\Theta}_{n, k}^{col}$ are connected. Since these subgraphs are isomorphic, we prove only for $\widehat{\Theta}_{n, k}^{row}$. Let $V_{S}^{row}$ and $V_{S'}^{row}$ be any two vertices of $\widehat{\Theta}_{n, k}^{row}$, where $S, S' \in {{[n]}\choose{k}}$. Let $S \setminus S' = \{i_1, \dots, i_m\}$ and $S' \setminus S = \{j_1, \dots, j_m\}$. Denote the sequence of subsets $S_t \in {{[n]}\choose{k}}$ for $t = 0, 1, \dots, m$ as follows.
		$$S_t = (S \setminus \{i_1, \dots, i_t\}) \cup \{j_1, \dots, j_t\}.$$
		Note that $S_0 = S$ and $S_m = S'$. By construction $|S_t \cap S_{t+1}| = k-1$ for all $t$, so $w(V_{S_t}^{row}, V_{S_{t+1}}^{row}) = n(k-1)$ and the vertices $V_{S_t}^{row}$, $V_{S_{t+1}}^{row}$ are connected by an edge in $\widehat{\Theta}_{n, k}^{row}$. This defines a path between $V_{S}^{row}$ and $V_{S'}^{row}$ in $\widehat{\Theta}_{n, k}^{row}$. Therefore, this subgraph is connected. Since $\Phi$ is an automorphism, it preserves connected components. This concludes the proof.
	\end{proof}

	\section{Main results} \label{sec4}
	The main goal of this section is to give a proof of Theorem \ref{mainth1}.
	Recall that the Hadamard (or entrywise) product of two matrices $A = (a_{i, j}), B = (b_{i, j}) \in \Mat_n(\F)$ is the matrix $(c_{i, j}) = C = A \circ B \in \Mat_n(\F)$ defined by $c_{i, j} = a_{i, j}b_{i, j}$. Denote by $R_{i}=\{\sum _{j=1}^n a_{ij}E_{ij}\vert a_{ij}\in \F\} \subseteq \Mat_{m\times n}(\F)$ the set of matrices  with all nonzero entries located in the row  $i$, and by  $C_{j}=\{\sum _{i=1}^m a_{ij}E_{ij}\vert a_{ij}\in \F\} \subseteq \Mat_{m\times n}(\F)$ the set of matrices with all   nonzero entries located in the column  $j$.
	
	The first step is the following lemma.
	
	\begin{lemma}\label{step0}
		Let $1 \leq k \leq n-1$ and let $T: \Mat_{n}(\F) \to \Mat_{n}(\F)$ be a bijective linear map satisfying $T(\Lambda^{\leq k}) \subseteq \Lambda^{\leq k}$. Then there exist permutations $\sigma_1, \sigma_2 \in S_n$ such that either
		\begin{equation}\label{eq01}
			T(R_i) = R_{\sigma_1(i)}, \;\; T(C_i) = C_{\sigma_2(i)} \;\;\; \mbox{for all} \;\;\; i \in [n],
		\end{equation}
		or
		\begin{equation}\label{eq02}
			T(R_i) = C_{\sigma_2(i)}, \;\; T(C_i) = R_{\sigma_1(i)} \;\;\; \mbox{for all} \;\;\; i \in [n].
		\end{equation}
	\end{lemma}
	\begin{proof}
		\textit{Case 1. $(k, n) \neq (2, 4)$.}
		By Lemma \ref{l2} the map $T$ induces an automorphism $T_*$ of the graph $\Theta_{n, k}$. Theorem \ref{autom} implies that either $T_*(\Theta_{n, k}^{row}) = \Theta_{n, k}^{row}$ and $T_*(\Theta_{n, k}^{col}) = \Theta_{n, k}^{col}$ or $T_*(\Theta_{n, k}^{row}) = \Theta_{n, k}^{col}$ and $T_*(\Theta_{n, k}^{col}) = \Theta_{n, k}^{row}$.
		
		Assume that $T_*(\Theta_{n, k}^{row}) = \Theta_{n, k}^{row}$ and $T_*(\Theta_{n, k}^{col}) = \Theta_{n, k}^{col}$. Let us fix $i \in [n]$.
		Denote $$\mathcal{D}_i = \left\{S \in {{[n]}\choose{k}} \; \bigg| \; i \in S \right\}.$$
		Since $k < n$ we have 
		$$R_i = \bigcap_{S \in \mathcal{D}_i} T(V_S^{row}).$$
		Since $T_*(\Theta_{n, k}^{row}) = \Theta_{n, k}^{row}$, we have 
		$$\{T(V_S^{row}) \; | \; S \in \mathcal{D}_i\} = \{V_{S'}^{row} \; | \; S' \in \mathcal{D}'_i\}$$
		for some $\mathcal{D}'_i \subseteq {{[n]}\choose{k}}$. Denote $W_i = \cap_{S' \in \mathcal{D}'_i} S'$.
		Therefore we obtain
		$$T(R_i) = \bigcap_{S \in \mathcal{D}_i} T(V_S^{row}) = \bigcap_{S' \in \mathcal{D}'_i} T(V_{S'}^{row}) = \bigoplus_{j \in W_i} R_j.$$
		Since $T$ is injective, we have $n = \dim(T(R_i)) = n|W_i|$, so $|W_i| = 1$. We can define $\sigma_1(i)$ by $W = \{\sigma_1(i)\}$. We obtain a map $\sigma_1: [n] \to [n]$ such that for each $i \in [n]$ we have $T(R_i) = R_{\sigma_1(i)}$. Since $T$ is injective, $\sigma_1$ is a permutation. Similarly, since $\Phi(\Theta_{n, k}^{col}) = \Theta_{n, k}^{col}$, we construct a permutation $\sigma_2 \in S_n$ such that $T(C_i) = C_{\sigma_2(i)}$ for all $i \in [n]$.
		
		Consider the second case where $T_*(\Theta_{n, k}^{row}) = \Theta_{n, k}^{col}$ and $T_*(\Theta_{n, k}^{col}) = \Theta_{n, k}^{row}$. Let $\tau: \Mat_n(\F) \to \Mat_n(\F)$ be the transposing map, i.e. $\tau(A) = A^T$. Then for each $S \in {{[n]}\choose{k}}$ we have $\tau(V_S^{row}) = V_S^{col}$ and $\tau(V_S^{col}) = V_S^{row}$.
		Therefore the map $(T \circ \tau)_*$ satisfies
		$(T \circ \tau)_*(\Theta_{n, k}^{row}) = \Theta_{n, k}^{row}$ and $(T \circ \tau)_*(\Theta_{n, k}^{col}) = \Theta_{n, k}^{col}$ and we can apply the previous case. There exist permutations $\sigma_1, \sigma_2 \in S_n$ such that 
		$$(T \circ \tau)(R_i) = R_{\sigma_1(i)}, \;\; (T \circ \tau)(C_i) = C_{\sigma_2(i)} \;\;\; \mbox{for all} \;\;\; i \in [n].$$
		Since $\tau(R_i) = C_i$ and $\tau(C_i) = R_i$, this immediately implies (\ref{eq02}).
		
		\textit{Case 2. $(k, n) = (2, 4)$.} In this case the map $T: \Mat_4(\F) \to \Mat_4(\F)$ also induces an automorphism $T_*$ of the graph $\Theta_{4, 2}$. However, this automorphism not necessary satisfies the statement of Theorem \ref{autom}, so we will study this case separately.
		
		Let us describe the graph $\Theta_{4, 2}$. It has $6$ vertices of the form $V_{S}^{row}$ and $6$ vertices of the form $V_{S}^{col}$, where $S \in {{[4]}\choose{2}}$. For each $S \in {{[4]}\choose{2}}$ we have
		$$w(V_{S}^{row}, V_{[4] \setminus S}^{row}) = w(V_{S}^{col}, V_{[4] \setminus S}^{col}) = 0,$$
		while the weights of all other edges is equal to $4$. 
		
		We have precisely $6$ edges of weight $0$ and $T_*$ acts on them as a permutation. Consider those $3$ of them which are formed by pairs of row subspaces, namely
		\begin{equation}\label{edges}
			(V^{row}_{\{1, 2\}}, V^{row}_{\{3, 4\}}), \;\; (V^{row}_{\{1, 3\}}, V^{row}_{\{2, 4\}}), \;\; (V^{row}_{\{1, 4\}}, V^{row}_{\{2, 3\}}).
		\end{equation}
		The image under $T_*$ of each of these edges has the form either $(V_{S}^{row}, V_{[4] \setminus S}^{row})$ or $(V_{S}^{col}, V_{[4] \setminus S}^{col})$. 
		
		If all the images have the form $(V_{S}^{row}, V_{[4] \setminus S}^{row})$, then we obtain $T_*(\Theta_{4, 2}^{row}) = \Theta_{4, 2}^{row}$ and, therefore, $T_*(\Theta_{4, 2}^{col}) = \Theta_{4, 2}^{col}$, so we can apply the argument from the case $1$ of the proof. If all the images have the form $(V_{S}^{col}, V_{[4] \setminus S}^{col})$, then we obtain $T_*(\Theta_{4, 2}^{row}) = \Theta_{4, 2}^{col}$ and $T_*(\Theta_{4, 2}^{col}) = \Theta_{4, 2}^{row}$, and we also can apply the argument from the case $1$.
		
		Let us consider the case where images of precisely $2$ edges among (\ref{edges}) have the form $(V_{S}^{row}, V_{[4] \setminus S}^{row})$. Without loss of generality we denote these edges by
		$$(V^{row}_{\{i, j\}}, V^{row}_{\{k, l\}}), \;\; (V^{row}_{\{i, k\}}, V^{row}_{\{j, l\}}),$$
		where $\{i, j, k, l\}$ is a permutation of $\{1, 2, 3, 4\}$. Denote the images of the vertices
		$V^{row}_{\{i, j\}}, V^{row}_{\{k, l\}}, V^{row}_{\{i, k\}}, V^{row}_{\{j, l\}}$
		by
		$V^{row}_{S_1}, V^{row}_{S_2}, V^{row}_{S_3}, V^{row}_{S_4}$, respectively.
		We obtain
		$$T(R_i) = T(V^{row}_{\{i, j\}} \cap V^{row}_{\{i, k\}}) = T(V^{row}_{\{i, j\}}) \cap T(V^{row}_{\{i, k\}}) = $$
		$$ = V^{row}_{S_1} \cap V^{row}_{S_3} = \bigoplus_{m \in S_1 \cap S_3} R_m.$$
		Since $T$ is an isomorphism, we obtain $|S_1 \cap S_3| = 1$, so $T(R_i) = R_m$ for some $m \in [4]$. After applying the same argument to $j, k$ and $l$ we obtain a map $\sigma_1: [4] \to [4]$ such that for each $i \in [4]$ we have $T(R_i) = R_{\sigma_1(i)}$. Since $T$ is injective, $\sigma_1$ is a permutation. Note that the images of precisely $2$ edges among 
		\begin{equation*}\label{edges2}
			(V^{col}_{\{1, 2\}}, V^{col}_{\{3, 4\}}), \;\; (V^{col}_{\{1, 3\}}, V^{col}_{\{2, 4\}}), \;\; (V^{col}_{\{1, 4\}}, V^{col}_{\{2, 3\}}).
		\end{equation*}
		have the form $(V_{S}^{col}, V_{[4] \setminus S}^{col})$. Therefore, similarly we obtain that there exists a permutation $\sigma_2 \in S_4$ such that $T(C_i) = C_{\sigma_2(i)}$ for all $i \in [4]$.
		
		Finally, we need to consider the case where image of precisely $1$ edge among (\ref{edges}) has the form $(V_{S}^{row}, V_{[4] \setminus S}^{row})$. Then if we consider the map $T \circ \tau$ of $T$, then the images under $(T \circ \tau)_*$ of precisely $2$ edges among (\ref{edges}) have the form $(V_{S}^{row}, V_{[4] \setminus S}^{row})$, so we can apply the previous result and obtain that there exist permutations $\sigma_1, \sigma_2 \in S_4$ such that 
		$$(T \circ \tau)(R_i) = R_{\sigma_1(i)}, \;\; (T \circ \tau)(C_i) = C_{\sigma_2(i)} \;\;\; \mbox{for all} \;\;\; i \in [4].$$ Therefore, we have $T(R_i) = C_{\sigma_2(i)}$ and $T(C_i) = R_{\sigma_1(i)}$ for all $i \in [4]$.
	\end{proof}

	\begin{lemma}\label{step1}
		Let $1 \leq k \leq n-1$ and let $T: \Mat_{n}(\F) \to \Mat_{n}(\F)$ be a bijective linear map satisfying $T(\Lambda^{\leq k}) \subseteq \Lambda^{\leq k}$. Then there exist permutations $\sigma_1, \sigma_2 \in S_n$ and a matrix $C \in \Mat_n(\F)$ consisting of nonzero elements such that either
		\begin{equation}\label{eq21}
			T(A) = C \circ \Bigl( P(\sigma_1) A P(\sigma_2) \Bigr) \;\;\; \mbox{for all} \;\;\; A \in \Mat_n(\F)
		\end{equation}
		or
		\begin{equation}\label{eq22}
			T(A) = C \circ \Bigl( P(\sigma_1) A^T P(\sigma_2) \Bigr) \;\;\; \mbox{for all} \;\;\; A \in \Mat_n(\F).
		\end{equation}
	\end{lemma}
	\begin{proof}
		Let us apply Lemma \ref{step0}. If (\ref{eq01}) holds, we obtain
		$$T \Bigl( \langle E_{i, j}\rangle \Bigr)  = T(R_i \cap C_j) = R_{\sigma_1(i)} \cap C_{\sigma_2(j)} = \langle E_{\sigma_1(i), \sigma_2(j)}\rangle.$$
		We have $\langle E_{i, j}\rangle \cong \langle E_{\sigma_1(i), \sigma_2(j)}\rangle \cong \F$, so there exists $c_{\sigma_1(i), \sigma_2(j)} \in \F$ such that $T(E_{i, j}) = c_{\sigma_1(i), \sigma_2(j)} E_{\sigma_1(i), \sigma_2(j)}$. Since $T$ is bijective, we have $c_{\sigma_1(i), \sigma_2(j)} \neq 0$. Consider the matrix $C = (c_{i, j}) \in \Mat_n(\F)$. The condition $T(E_{i, j}) = c_{\sigma_1(i), \sigma_2(j)} E_{\sigma_1(i), \sigma_2(j)}$ is equivalent to formula (\ref{eq21}), so the lemma is proved it this case.
		
		If (\ref{eq02}) holds, we have 
		$$T \Bigl( \langle E_{i, j}\rangle \Bigr)  = T(R_i \cap C_j) = C_{\sigma_2(i)} \cap R_{\sigma_1(j)} = \langle E_{\sigma_1(j), \sigma_2(i)}\rangle.$$
		Similarly we obtain $T(E_{i, j}) = c_{\sigma_1(j), \sigma_2(i)} E_{\sigma_1(j), \sigma_2(i)}$ for some nonzero $c_{i, j} \in \F$, which is equivalent to formula (\ref{eq22}). This concludes the proof.
	\end{proof}
	
	Recall that for $A \in \Mat_{n}(\F)$ and $J_1,J_2\subseteq [n]$, we denote by $A[J_1\vert J_2]$ the $|J_1|\times |J_2|$ submatrix of $A$ lying on the intersection of the rows with the indices from $J_1$ and the columns with the indices from $J_2$.
	The next step is as follows.
	
	\begin{lemma}\label{step2}
		Let $1 \leq k \leq n-1$ and let $C = (c_{i, j}) \in \Mat_n(\F)$ be a matrix consisting of nonzero elements such that for each $A \in \Lambda^{\leq k}$ we have $C \circ A \in \Lambda^{\leq k}$. Then there exist nonzero elements $d_{1, 1} \dots, d_{1, n}, d_{2, 1}, \dots, d_{2, n} \in \F$ such that $c_{i, j} = d_{1, i} d_{2, j}$, i.e. $C = d_1^T d_2$, where $d_1 = (d_{1, 1} \dots, d_{1, n})^T \in \F^n$ and $d_2 = (d_{2, 1} \dots, d_{2, n})^T \in \F^n$.
	\end{lemma}
	\begin{proof}
		Let us show that $\rk C = 1$. It is sufficient to prove that for any $i, m, j, l$ with $i \neq m$ and $j \neq l$ we have $c_{i, j}c_{m, l} - c_{i, l}c_{m, j} = 0$. Let $r_1, \dots, r_{k-1} \in [n] \setminus \{i, m\}$ be pairwise distinct and let $s_1, \dots, s_{k-1} \in [n] \setminus \{j, l\}$ be pairwise distinct. Consider the matrix
		$$X = E_{i, j} + E_{i, l} - E_{m, j} + E_{m, l} + \sum_{t=1}^{k-1}E_{r_t, s_t} \in \Mat_n(\F).$$
		We claim that $\prk(X) = k$. Indeed, the only $(k+1) \times (k+1)$-submatrix we need to check is $X[r_1, \dots, r_{k-1}, i, m \vert s_1, \dots, s_{k-1}, j, l]$, because all other rows and columns are zero. Since the submatrix $X[r_1, \dots, r_{k-1} \vert s_1, \dots, s_{k-1}]$ is a permutation matrix, we have
		$$ \per X[r_1, \dots, r_{k-1}, i, m \vert s_1, \dots, s_{k-1}, j, l] = \per (E_{i, j} + E_{i, l} - E_{m, j} + E_{m, l}) = 0.$$
		
		Since $X \in \Lambda^{\leq k}$ it follows that $C \circ X \in \Lambda^{\leq k}$. In particular, we obtain
		$$ 0 = \per (C \circ X)[r_1, \dots, r_{k-1}, i, m \vert s_1, \dots, s_{k-1}, j, l] = $$ 
		$$ = \Bigl( \prod_{t=1}^{k-1} c_{r_t, s_t} \Bigr) \per (c_{i, j}E_{i, j} + c_{i, l}E_{i, l} - c_{m, j}E_{m, j} + c_{m, l}E_{m, l}) = $$
		$$ = \Bigl( \prod_{t=1}^{k-1} c_{r_t, s_t} \Bigr)\Bigl( c_{i, j}c_{m, l} - c_{i, l}c_{m, j}\Bigr).$$
		Since $C$ consists of nonzero elements, we obtain $c_{i, j}c_{m, l} - c_{i, l}c_{m, j} = 0$.
		
		Therefore $\rk C \leq 1$ and there exist $d_1 = (d_{1, 1} \dots, d_{1, n})^T \in \F^n$ and $d_2 = (d_{2, 1} \dots, d_{2, n})^T \in \F^n$ such that $C = d_1^T d_2$. Since all the elements of $C$ are nonzero, it follows that $d_{1, 1} \dots, d_{1, n}, d_{2, 1}, \dots, d_{2, n} \in \F$ are also nonzero.
	\end{proof}
	
	Now we are ready to prove the main result.
	\begin{proof}[Proof of Theorem \ref{mainth1}]
		By Lemma \ref{step1} there exist permutations $\sigma_1, \sigma_2 \in S_n$ and a matrix $C \in \Mat_n(\F)$ consisting of nonzero elements such that either (\ref{eq21}) or (\ref{eq22}) holds. By Lemma \ref{if1} the maps $A \mapsto P(\sigma_1) A P(\sigma_2)$ and $P(\sigma_1) A^T P(\sigma_2)$ preserve the set $\Lambda^{\leq k}$, so the matrix $C$ satisfies $C \circ A \in \Lambda^{\leq k}$ whenever $A \in \Lambda^{\leq k}$. Therefore, by Lemma \ref{step2}, there exist nonzero elements $d_{1, 1} \dots, d_{1, n}, d_{2, 1}, \dots, d_{2, n} \in \F$ such that $c_{i, j} = d_{1, i} d_{2, j}$. In particular, for each $A \in \Mat_n(\F)$ we have $C \circ A = D_1 A D_2$, where $D_1 = \diag(d_{1, 1}, \dots, d_{1, n})$ and $D_2 = \diag(d_{2, 1}, \dots, d_{2, n})$. Hence formulas (\ref{eq21}) and (\ref{eq22}) imply formulas (\ref{eq1}) and (\ref{eq2}), respectively. This concludes the proof.
	\end{proof}

\end{document}